\documentclass[11pt]{amsart}

\usepackage{amsmath}
\usepackage{amsfonts}
\usepackage{amssymb}
\usepackage{amsthm}
\usepackage{indentfirst}
\usepackage[T1]{fontenc}

\usepackage{tikz}
\usepackage{verbatim}
\usetikzlibrary{matrix}

\usepackage{enumerate}

\newtheorem{theorem}{Theorem}[section]
\newtheorem{lem}[theorem]{Lemma}
\newtheorem{prop}[theorem]{Proposition}
\newtheorem{cor}[theorem]{Corollary}
\newtheorem{defn}[theorem]{Definition}

\newcommand{\bb}{{\bf b}}
\newcommand{\bc}{{\bf c}}
\newcommand{\bg}{{\bf g}}

\newcommand{\bba}{{\mathbb{A}}}
\newcommand{\bbp}{{\mathbb{P}}}
\newcommand{\bbq}{{\mathbb{Q}}}
\newcommand{\bbr}{{\mathbb{R}}}
\newcommand{\bbz}{{\mathbb{Z}}}

\newcommand{\Spec}{\operatorname{Spec}\,}

\title{Waring's problem for rational functions in one variable}

\author[Bo-Hae Im]{Bo-Hae Im}
\address{Department of Mathematical Sciences, KAIST, 291 Daehak-ro, Yuseong-gu, Daejeon, 34141, South Korea}
\email{bhim@kaist.ac.kr}


\author[Michael Larsen]{Michael Larsen}
\address{Department of Mathematics, Indiana University, Bloomington,
Indiana 47405, USA}
\email{mjlarsen@indiana.edu}
\date{\today}
\thanks{Bo-Hae Im was supported by Basic Science Research Program through the National Research Foundation of Korea(NRF) funded by the Ministry of  Science, ICT \& Future Planning(NRF-2017R1A2B4002619). Michael Larsen was partially supported by the National Science Foundation.}

\begin{document}

\maketitle
\begin{abstract}
Let $f\in \bbq(x)$ be a non-constant rational function.  We consider ``Waring's Problem for $f(x)$,'' i.e., whether every element of $\bbq$ can be written as a bounded sum of elements of $\{f(a)\mid a\in \bbq\}$.  For rational functions of degree $2$, we give necessary and sufficient conditions.  For higher degrees, we prove that every polynomial of odd degree and every odd Laurent polynomial satisfies Waring's Problem.  We also consider the ``Easier Waring's Problem'': whether every element of $\bbq$ can be  represented as a bounded sum of elements of $\{\pm f(a)\mid a\in \bbq\}$.
\end{abstract}

\section{Introduction}
The classical Waring's Problem (WP) asks if, for every positive integer $d$, there exists $N$ such that every natural number can be written as the sum of $N$ $d$th powers of natural numbers.  This was settled in the affirmative by Hilbert \cite{Hilbert}.  Shortly afterward, Erich Kamke \cite{Kamke} proved that for every polynomial $f(x)\in \bbz[x]$ with positive leading coefficient there exists $N$ such that every sufficiently large integer satisfying an obvious congruence condition (depending on $f(x)$) can be written as a sum of $N$ values  of the form $f(x_i)$, where the $x_i$ are natural numbers.

In this paper, we propose to consider the analogous problem for rational functions.  Since in this setting, we can in general only expect $f(x_i)$ to belong to $\bbq$, we consider the question of whether every rational number, or every sufficiently positive rational number, can be written as a sum of values $f(x_i)$, $x_i\in \bbq$.

In 1934, Edward Wright \cite{Wright} introduced the Easier Waring's Problem (EWP): to represent an integer as a sum or difference of a fixed number of $d$th powers,
i.e., as $\pm x_1^d\pm x_2^d \cdots \pm x_N^d$.  We consider both the original WP and the EWP for rational functions.  We give necessary and sufficient conditions for solubility of these problems, but these conditions are quite far apart.  For most rational functions, we do not know whether either version of WP is soluble.

This paper was written simultaneously with the paper \cite{LN} of the second named author and Dong Quan Ngoc Nguyen on Waring's Problem for unipotent algebraic groups over number fields.  Since a basic idea behind that paper is that the proper setting for Waring-type problems is polynomial-valued maps,  the fact that one can prove such results for Laurent polynomials came as something of a surprise.

The second named author would like to acknowledge useful conversations with Nguyen related to the subject of this paper.

\section{Generalities}

Throughout this section, $X$ denotes a subset of  $\bbq$.
\begin{defn}
We say $X$ is a \emph{base} (resp. \emph{positive base}, \emph{negative base}, or \emph{open base}) if for some positive integer $N$,
$$\underbrace{X+X+\cdots+X}_N$$
is all of $\bbq$ (resp. contains $(a,\infty)\cap \bbq$ for some $a$, contains $(-\infty,b)\cap \bbq$ for some $b$, or contains $(a,b)\cap \bbq$ for some $a<b$).
We say $X$ is a \emph{virtual base} if
$$\underbrace{\pm X\pm X\pm \cdots\pm X}_N = \bbq.$$
\end{defn}

Clearly, all of these properties are invariant under translation of $X$ or multiplication of $X$ by any positive rational scale factor.
The following properties are also immediate:

\begin{lem}
\label{obvious}
For $X\subseteq \bbq$, we have
\begin{enumerate}[$(a)$]
\item If $X$ is a base, it is both a positive base and a negative base.
\item If $X$ is a positive base or a negative base, then it is both an open base and a virtual base.
\item If $X$ is a positive base and unbounded below or a negative base and unbounded above, then it is a base.
\item If $X$ is a base (resp. positive base, negative base, open base, or virtual base), then $-X$ is a base
(resp. negative base, positive base, open base, or virtual base).
\end{enumerate}

\end{lem}

The following lemma gives obvious obstructions to a set $X$ being a base (or positive base, etc.)

\begin{lem}
\label{bounds}
For $X\subseteq \bbq$, we have
\begin{enumerate}[$(a)$]
\item If $X$ is a positive base, it cannot be bounded above.
\item If $X$ is a negative base, it cannot be bounded below.
\item If $X$ is a virtual base, it cannot be bounded.
\item If $X$ is any kind of base, it cannot be $p$-adically bounded for any prime $p$.
\end{enumerate}

\end{lem}

If $f(x)\in \bbq(x)$ is a rational function and $F$ is any field of characteristic $0$, we denote by $f(F)$ the set of values $f(a)$ as $a$ ranges over all elements of $F$ which are not poles of $f$.

\begin{defn}
We say $f$ \emph{satisfies WP} if $f(\bbq)$ is a base.  We say it \emph{satisfies the positive} (resp. \emph{negative}) \emph{WP} if
$f(\bbq)$ is a positive (resp. negative) base.   We say it \emph{satisfies the EWP} if $f(\bbq)$ is a virtual base.
\end{defn}

\begin{prop}
\label{poles}
If $f(x)\in \bbq(x)$ is a rational function then
\begin{enumerate}[$(a)$]
\item For $f$ to satisfy WP, it is necessary for it to have at least two distinct  poles in $\bbr\bbp^1$ or one  pole of odd order in  $\bbr\bbp^1$
\item If $f(\bbq)$ is an open base and $f$ has at least one pole of odd order in  $\bbr\bbp^1$, then $f$ satisfies WP.
\item For $f$ to satisfy the EWP, it is necessary for $f$ to have a  pole in  $\bbr\bbp^1$
\item For $f$ to satisfy the EWP, it is necessary that for each prime $p$, $f$ has a pole in $\bbq_p\bbp^1$.
\end{enumerate}

\begin{proof}
Part (a) follows from parts (a) and (b) of Lemma~\ref{bounds}.  Parts (c) and (d) follow from parts (c) and (d) of Lemma~\ref{bounds} respectively.
For part (b), we note that if  $f$ has at least one pole of odd order in  $\bbr\bbp^1$, then the closure of $f(\bbq)$ contains a neighborhood of $\infty$ in $\bbr\bbp^1$, i.e., contains all real numbers of absolute value $> B$ for some $B$.  If
$$(a,b)\cap\bbq \subset \underbrace{f(\bbq)+\cdots+f(\bbq)}_N$$
and $M(b-a) > 2B$, then setting $Y := f(\bbq)\cap  (-\infty,-B)$ and $Z:= f(\bbq)\cap  (B,\infty)$, we have
$$\bbq\subseteq ((Y \cup Z )+(Ma,Mb)\cap \bbq)\subseteq\underbrace{f(\bbq)+\cdots+f(\bbq)}_{1+MN}.$$

\end{proof}

\end{prop}

The following proposition shows that the property of being a base is not affected by any finite subset of elements.

\begin{prop}
\label{change}
If $X, Y\subseteq \bbq$ and $X\setminus Y$ and $Y\setminus X$ are finite, then $X$ is a base if and only if $Y$ is a base.
\end{prop}

\begin{proof}
Without loss of generality, we may assume $Y = X\cup\{y\}$.  Translating, we may assume $y=0$.  The non-trivial direction is that if $Y$ is a base, the same is true for $X$.
Let
$$X_n := \underbrace{X+X+\cdots+X}_n, Y_n := \underbrace{Y+Y+\cdots+Y}_n.$$
As $Y$ is a base, $Y_N=\bbq$ for some $N>0$.  Let $x\in X$ be any non-zero element.  For any positive integer $m$, $-x/m\in Y_N$, so $-x/m \in X_i$ for some positive
integer $i\le N$, and this implies $0 = x+m(-x/m) \in X_{1+im}$.  Letting $M:= 1+im$, and applying the same reasoning to $-x/M\in Y_N$, we see that $0\in X_{1+jM}$
for some positive integer $j$.  The set of positive integers $k$ such that $0\in X_k$ is a semigroup, and it contains the relatively prime elements $M$ and $1+jM$, so it
contains all integers $\ge K$ for some integer $K$. Thus,
$$X_{K+N}\supset \{0\}\cup X_1\cup\cdots\cup X_N = Y_N = \bbq.$$

\end{proof}

\begin{cor}
\label{compose}
If $g(x)\in \bbq(x)$ is a fractional linear transformation, then
$f(g(F))$ is a base if and only if $f(F)$ is a base.
\end{cor}

\begin{proof}
As $g(F)\setminus F$ and $F\setminus g(F)$ have at most one element each, the corollary follows immediately.
\end{proof}

\section{The EWP for Laurent polynomials in  characteristic $0$}

Throughout this section, $K$ will denote a field of characteristic $0$.  Our main result  is the following theorem.

\begin{theorem}\label{main-Laurent} If $f(x)\in K\left[x,\frac{1}{x}\right]$ is a non-constant Laurent polynomial, then there exists a positive integer $N$ such that
$$\underbrace{\pm f( K)\pm f( K)\pm \cdots \pm f( K)}_{N} = K.$$
In particular, $f(x)$ satisfies the EWP.
\end{theorem}

In fact, we prove the following stronger result.

\begin{theorem}\label{main-stronger} Let $S$ be a finite set of non-zero integers.  Let
$ K^S$ denote the $ K$-algebra of functions $S\to  K$ and
$\bg^S: K^*\rightarrow  K^S$ denote the function defined by
$$\bg^S(x):= s\mapsto x^s.$$
Let
$$X_k^S:=\underbrace{ \bg^S( K^*)+  \cdots + \bg^S( K^*)}_{k}\underbrace{-\bg^S( K^*)-  \cdots - \bg^S( K^*)}_{k}.$$
Then, there exists a positive integer $N$ such that
    $$X_N^S = K^S.$$
\end{theorem}

We defer the proof of Theorem~\ref{main-stronger}, starting instead with the following proposition:

\begin{prop} Theorem~\ref{main-stronger} implies Theorem~\ref{main-Laurent}.
\end{prop}
\begin{proof}
For a given $f(x)\in  K\left[x,\frac{1}{x}\right]$, let $S$ denote the set of non-zero exponents of monomials occurring in $f(x)$.
As $f(x)$ is not constant, $S$ is non-empty.  We write
$$f(x)=a_0 + \sum\limits_{s\in S} a_s x^s,$$
where $a_0$ may be zero but $a_s \neq 0$ for all $s\in S$.
For $x_i, y_j\in K^*$,
$$\sum\limits_{i=1}^Nf(x_i)-\sum\limits_{j=1}^Nf(y_j)=c,$$
if and only if
\begin{equation}
\label{want}\sum\limits_{s\in S} a_s \left(\sum\limits_{i=1}^Nx_i^s-\sum\limits_{j=1}^Ny_j^s\right)=c.
\end{equation}

Choose $c_s\in K$ for each $s\in S$ such that
\begin{equation}
\label{dot-prod}
\sum\limits_{s\in S} a_s c_s=c.
\end{equation}
If we assume that Theorem~\ref{main-stronger} is true for some positive integer $N$, then there exist $x_i, y_j\in K^*$ such that for all $s\in S$,
\begin{equation}
\label{c-control}
c_s =\sum\limits_{i=1}^Nx_i^s-\sum\limits_{j=1}^N y_j^s.
\end{equation}
Thus, (\ref{want}) follows from (\ref{dot-prod}) and (\ref{c-control}).
This proves Theorem~\ref{main-Laurent}.
\end{proof}

To prove Theorem~\ref{main-stronger}, we need the following lemma.

\begin{lem}\label{operation} Following the notations in Theorem~\ref{main-stronger},
\begin{enumerate}[$(a)$]
    \item $X_k^S+X_\ell^S=X_{k+\ell}^S$.
    \item  $X_k^S-X_\ell^S=X_{k+\ell}^S$.
    \item $X_k^SX_\ell^S\subseteq X_{2k\ell}^S$.
\end{enumerate}
\end{lem}

\begin{proof}
Parts (a) and (b) are trivial. For (c),
\begin{align*}
&\left(\sum\limits_{i=1}^k\bg^S(x_i)-\sum\limits_{j=1}^k\bg^S(y_j) \right)\left(\sum\limits_{r=1}^\ell \bg^S(w_r)-\sum\limits_{s=1}^\ell \bg^S(z_s)\right)\\
&\hspace{1.5 cm}=\left(\sum\limits_{i=1}^k\sum\limits_{r=1}^\ell \bg^S(x_i w_r)+\sum\limits_{j=1}^k\sum\limits_{s=1}^\ell \bg^S(y_j z_s)\right)
\\&\hspace{2.5 cm}-\left(\sum\limits_{j=1}^k\sum\limits_{r=1}^\ell \bg^S(y_j w_r)+\sum\limits_{i=1}^k\sum\limits_{s=1}^\ell \bg^S(x_i z_s)\right),
\end{align*}
which is in $X_{2k\ell}^S$.
\end{proof}
\

\begin{proof}[Proof of Theorem~\ref{main-stronger}.] We use induction on $|S|$.

If $|S|=1$, i.e. $S = \{s\}$, $s\neq 0$, and $\{a^s\mid a\in K^*\} = \{a^{-s}\mid a\in K^*\}$,
so without loss of generality we may assume $s$ is a positive integer.
Let $\delta$ denote the difference operator, so $(\delta f)(x) = f(x+1) - f(x).$
By induction on $s$, we see
$$(\delta^{s-1}f)(x) = s!\,x+ \frac{(s-1) s!}2.$$
Thus,
$$s!\,x+ \frac{(s-1) s!}2 = \sum_{i=0}^{s-1} (-1)^i \binom {s-1}i (x+s-1-i)^s \in X_{2^{s-2}}^{S},$$
as long as $x\not\in \{0,-1,-2,\ldots,1-s\}$.  In particular, $K\setminus \bbz\subset X_{2^{s-2}}^{S}$.
Since $X_1^S$ is not contained in a single $\bbz$-coset of $K$ (for instance, it contains $0$ and $1-(1/2)^s$),
it follows that
$$K\subseteq X_1^S+X_{2^{s-2}}^{S} = X_{1+2^{s-2}}^{S}.$$

Suppose the claim of the theorem is true for some $S$, and let us prove it for $S\cup\{t\}$.  We know that for some $N$,
$$X_N^{S}= K^{S} \text{ and } X_N^{\{t\}}= K.$$

Let $p_k^{S,t}\colon X_i^{S\cup\{t\}}\to X_k^S$ and $q_k ^{S,t}\colon X_i^{S\cup\{t\}}\to X_k^{\{t\}}$ denote projection maps.  In particular, $p_k^{S,t}$
and $q_k ^{S,t}$ are surjective for $k\ge N$.

Then we have two cases:

Case 1. There exist $v, w\in X_{N+1}^{S\cup\{t\}}$ such that $v\neq w$ but
$$p_{N+1}^{S,t}(v)=p_{N+1}^{S,t}(w).$$

Case 2. No such $v, w\in X_{N+1}^{S\cup\{t\}}$ exist, or (equivalently), $p_{N+1}^{S,t}$ is injective (and therefore bijective).

\

In Case 1, $0\neq v-w\in X_{2N+2}^{S\cup\{t\}}$ and $p_{2N+2}^{S,t}(v-w)=0$.  Regarding $v-w$ as a function $f\colon S\cup\{t\}\to  K$, we have $f(s)=0$ for all $s\in S$ but $f(t)\neq 0$..
As $q_N^{S,t}$ maps onto $ K$, multiplying $X_N^{S\cup\{t\}}$ by $v-w\in X_{2N+2}^{S\cup\{t\}}$, the set $X_{2N(2N+2)}^{S\cup\{t\}}$ contains
all functions $S\cup\{t\}\to  K$ which vanish identically on $S$.
Since $X_N^{S}= K^{S}$, it follows that
$$X_{2N(2N+2)+N}^{S\cup\{t\}}= K^{S\cup\{t\}}.$$

It therefore suffices to consider Case 2.  Now $X_N^{S\cup\{t\}}\subseteq X_{N+1}^{S\cup\{t\}}$.
If equality does not hold, since $p_N^{S,t}$ is surjective, there exist $v\in X_N^{S\cup\{t\}}$ and $w\in X_{N+1}^{S\cup\{t\}}\setminus X_N^{S\cup\{t\}}$ with $p_{N+1}^{S,t}(v)=p_{N+1}^{S,t}(w)$.
This is impossible in Case 2.  Thus, $X_N^{S\cup\{t\}}= X_{N+1}^{S\cup\{t\}}$. Then, by Lemma~\ref{operation}, we have that
$$X_N^{S\cup\{t\}}+X_1^{S\cup\{t\}}=X_{N+1}^{S\cup\{t\}}=X_N^{S\cup\{t\}},$$
so by induction we have
\begin{align}\label{all}
\notag    X_k^{{S\cup\{t\}}}=X_N^{S\cup\{t\}}, \text{ for all } k\geq N.
\end{align}

We set $X^{S\cup\{t\}}:=X_N^{S\cup\{t\}}$ and write $p^{S,t}$ (resp. $q^{S,t}$) for the projection map $p_N^{S,t}$ (resp. $q_N^{S,t}$).
By Lemma~\ref{operation}, $X^{S\cup\{t\}}$ is closed under addition, subtraction and multiplication, i.e., it is a (possibly non-unital) ring.
As we are in Case 2,  $p^{S,t}$
is an isomorphism of non-unital rings, and it follows that $X^{S\cup\{t\}}$ is a unital ring and $p^{S,t}$ is an isomorphism of unital rings.

Now, $q^{S,t}\colon X^{S\cup\{t\}} \to X^{\{t\}} =  K$  is  a surjective non-unital homomorphism of rings and therefore a ring homomorphism.  Composing it  with the inverse of $p^{S,t}$, we obtain a
ring homomorphism  $\phi:  K^{S}\rightarrow  K$ which expresses the value of $f\in X^{S\cup\{t\}}$ at $t$ in terms of the restriction of $f$ to $S$.
Letting $e_s$ denote the idempotent of $ K^S$ which is $0$ on $S\setminus \{s\}$ and $1$ on $s$,
$\phi(e_s)\in \{0,1\}$ for all $s$, and as $\phi$ maps the multiplicative identity $\sum_{s\in S} e_s$ to $1$, it follows that $\phi(e_s) = 1$ for some $s\in S$.
This implies that $\phi(f) = \phi(f e_s)$ for all $f\in  K^S$, i.e., that the homomorphism $\phi$ factors through the projection $ K^S\to  K$ given by evaluation at $s$.
Thus, there exists an endomorphism $\psi$ of $K$  such that  $\phi(f) = \psi(f(s))$, and $X^{S\cup\{t\}}$ must consist of all functions $f\colon  S\cup\{t\}\to  K$ such that $f(t)=\psi(f(s))$.  This is absurd because, for instance, $f := \bg^{S\cup\{t\}}(2) - \bg^{S\cup\{t\}}(1)\in X^{S\cup\{t\}}$ satisfies
$$2^t-1 = f(t) = \psi(f(s)) = \psi(2^s-1) = 2^s-1$$
although $t>s$.
This completes the proof.

\end{proof}

\begin{cor}
\label{odd}
If $f(x)$ is an odd Laurent polynomial, then $f(x)$ satisfies WP.
\end{cor}
\section{Waring's problem for polynomials over $\bbq$}

The main result in this section is the following:

\begin{theorem}
\label{Main-Poly}
Let $f(x)\in \bbq[x]$ be a non-constant polynomial.  If $f$ is of odd degree, it satisfies WP.  If $f$ is of even degree, $f(\bbq)$ is a positive base or a negative base, according to whether the leading coefficient of $f$ is positive or negative.
\end{theorem}

Let $\bg\colon \bbq\to \bbq^d$ denote the map $\bg(x) = (x,x^2,\ldots,x^d)$.  We begin with the following lemma:

\begin{prop}
\label{smooth}
Let $m$, $d$, and $r$ denote positive integers, $m\ge d+r$.   Let $\mathbf{a} = (a_1,\ldots,a_m)\in \bbr^m$, and let
$\mathbf{F} = (F_1,\ldots,F_r)$ denote an $r$-tuple of linear forms in $\mathbf{x} = (x_1,\ldots,x_m)$.
For fixed $\mathbf{a}$ satisfying
\begin{equation}
\label{d-different}
\bigm|\!\{a_1,\ldots,a_m\}\!\bigm| \ge d,
\end{equation}
the set of $\mathbf{F}$
such that the
morphism $\bba^m\to \bba^{d+r} = \bba^d\times \bba^r$ given by
$$\mathrm{x}\mapsto \bigl(\sum_i \bg(x_i),\mathbf{F}(\mathrm{x})\bigr)$$
is smooth at $\mathbf{a}$ forms a dense open subset of the variety $\bba^{rm}$ of $r$-tuples of linear forms in $m$ variables.
\end{prop}

\begin{proof}
By (\ref{d-different}), without loss of generality we may assume that $a_1,\ldots,a_d$ are pairwise distinct.
Thus the Vandermonde determinant
$$\det\begin{pmatrix}
1&1&\cdots&1 \\
a_1&a_2&\cdots&a_d \\
\vdots&\vdots&\ddots&\vdots \\
a_1^{d-1}&a_2^{d-1}&\cdots &a_d^{d-1} \\
\end{pmatrix}.$$
is non-zero, and
$$\begin{pmatrix}
1&1&\cdots&1 \\
a_1&a_2&\cdots&a_m \\
\vdots&\vdots&\ddots&\vdots \\
a_1^{d-1}&a_2^{d-1}&\cdots &a_m^{d-1} \\
\end{pmatrix}
$$
has rank $d$.  It follows that for a generic choice of $r\times m$ matrices $b_{ij}$,
the matrix
$$\begin{pmatrix}
1&1&\cdots&1 \\
a_1&a_2&\cdots&a_m \\
\vdots&\vdots&\ddots&\vdots \\
a_1^{d-1}&a_2^{d-1}&\cdots &a_m^{d-1} \\
b_{11}&b_{12}&\cdots&b_{1m} \\
\vdots&\vdots&\ddots&\vdots \\
b_{r1}&b_{r2}&\cdots&b_{rm} \\
\end{pmatrix}
$$
has rank $d+r$.  The proposition now follows from the Jacobian condition for smoothness.
\end{proof}

\begin{proof}[Proof of Theorem~\ref{Main-Poly}]
Let $d := \deg f$, and let $m$ be an integer greater than $d$.
Fix pairwise distinct  rational numbers $a_1,\ldots,a_m$, and let $b_j = \sum_i a_i^j$.  Let $\bb$ denote the vector $(b_1,\ldots,b_d)$.  Consider the closed subscheme $V$ of $\bbp^m$ over $\Spec \bbq[t_1,\ldots,t_d]$
defined by the system of $d$ homogeneous
equations
\begin{equation}
\begin{split}
\label{generic}
X_1+\cdots+X_m&= t_1X_0\\
X_1^2+\cdots+X_m^2&= t_2X_0^2 \\
\vdots\\
X_1^d+\cdots+X_m^d &= t_d X_0^d.\\
\end{split}
\end{equation}
Let $V^{\bc}$ denote the fiber of $V$ over $\bc = (c_1,\ldots,c_d)$.

Using the Jacobian criterion for smoothness and the Vandermonde determinant as before, $(X_0:\cdots:X_m)$ is a non-singular point
of $V^{\bc}$
as long as there are at least $d+1$ distinct values among $X_0,\ldots,X_m$.  Thus, for all $\bc$, the singular locus of $V^{\bc}$ has dimension at most $d-1$,
and $(1:a_1:\cdots:a_m)$ is a non-singular point of $V^{\bb}$.

By Bertini's theorem as formulated by Zariski \cite{Z}, the intersection of $V^{\bc}$ with a generic hyperplane $G_1 = 0$ in $\bbp^{m}$
can be singular only at a subvariety of
$V^{\bc}$ of dimension less than that of $\mathrm{Sing}( V^{\bc})$.  We may choose $G_1$ to have coefficients in $\bbq$ since the rational hyperplanes are dense in the projective space of all real hyperplanes in the real topology and therefore in the Zariski topology.  Iteratively choosing $G_2,\ldots, G_d$ generically, the intersection $W_{\mathbf{G}}^{\bc}$
of $V^{\bc}$ with the locus $G_1=\cdots=G_d=0$ is non-singular.

Writing $G_i = b_{i0}X_0+b_{i1}X_1+\cdots+b_{im}X_m$, we define $F_i := b_{i1}x_1+\cdots+b_{im}x_m$, where the $x_i := X_i/X_0$ are affine coordinates
on the affine open subset $\bba^m$ of $\bbp^m$ given by $X_0\neq 0$.  Fixing $b_{ij}$ for $j\ge 1$ and letting $b_{i0}$ vary, for a generic $d$-tuple
$(b_{10},\ldots,b_{d0})$ and generic $\bc$, $W_{\mathbf{G}}^{\bc}$ is non-singular; the complement of $X_0=0$ is then given by the system of equations
\begin{equation*}
\begin{split}
\label{generic}
x_1+\cdots+x_m&= c_1\\
x_1^2+\cdots+x_m^2&= c_2 \\
\vdots\\
x_1^d+\cdots+x_m^d &= c_d\\
F_1(x_1,\ldots,x_m) &= -b_{10} \\
\vdots\\
F_d(x_1,\ldots,x_m) &= -b_{d0} \\
\end{split}
\end{equation*}
Applying Proposition~\ref{smooth} and the implicit function theorem, we conclude that there is a $d$-tuple $\mathrm{G}$ and a non-empty open set $U\subset \bbr^d$
such that for all $\bc\in U$, $W^{\bc}_{\mathbf{G}}$ has a real point.

If $\dim W^{\bc}_{\mathbf{G}} = m-2d$ is large enough, by a theorem of Brauer \cite{Brauer}, $W^{\bc}_{\mathbf{G}}(\bbz_p) = W^{\bc}_{\mathbf{G}}(\bbq_p)$ is non-empty.  By a theorem of Timothy Browning and Roger Heath-Brown \cite[Theorem 1.7]{BH}, this implies that $W^{\bc}_{\mathbf{G}}(\bbq)\subset V^{\bc}(\bbq)$ is non-empty.

If $f(x) = a_0+a_1 x + \ldots+a_d x^d$,
$$\{a_0m+a_1 c_1 + \cdots + a_d c_d\mid \bc \in U\cap \bbq^d\}\subset \underbrace{f(\bbq)+\cdots+f(\bbq)}_m.$$
Thus $f(\bbq)$ is an open base.  If $d$ is odd, then $f(\bbq)$ is dense in $(-\infty, -B)\cup (B,\infty)$ for some $B$, and it follows as in the proof of Proposition~\ref{poles} that $f(\bbq)$ is a base.  If $d$ is even and $a_d>0$ (resp. $a_d < 0$), then $f(\bbq)$ is dense in $(B,\infty)$ (resp. $(-\infty, -B)$) for some $B$, and it follows that
$f(\bbq)$ is a positive (resp. negative) base.

\end{proof}

\begin{theorem}
If $f(x)\in\bbq(x)$ is a rational function of degree $2$, then we have the following:
\begin{enumerate}[(a)]
\item If $f$ has two distinct poles in $\bbq\bbp^1$, it satisfies WP.
\item If $f$ has one pole in $\bbq\bbp^1$,  it satisfies the EWP but not WP.
\item  If $f$ has no pole in $\bbq\bbp^1$, it does not satisfy even the EWP.
\end{enumerate}

\end{theorem}

\begin{proof}
Let $g$ be a fractional linear transformation over $\bbq$ mapping $0$ and $\infty$ to the two poles of $f$.  Then $f(g(x))$ is a rational function of degree two with poles $0$ and $\infty$ and must therefore be of the form $\frac{ax^2+bx+c}{x}$.  Let $h(x) = f(g(x))-b = ax + c/x$.  By Corollary~\ref{odd}, there exists $N$ such that every rational number is a sum of $N$ terms in $h(\bbq)$.  It follows that every rational number is a sum of $N$ terms of $f(g(x))$.  Thus $h(\bbq)$ is a base, and by Corollary~\ref{compose}, the same is true of $f(\bbq)$.

For part (b), let $g$ be a fractional linear transformation over $\bbq$ mapping $\infty$ to the pole of $f$ (which must be double).  Then $h(x) = f(g(x)) = ax^2+bx+c$ for some $a,b,c\in \bbq$.  Rescaling and translating, we may assume $h(x)$ is of the form $x^2+d$.  Now, $h(\bbq)$ is bounded below, so  $h(\bbq)$ is not a base, and therefore that $f(\bbq)$ is not a base.  On the other hand, every value of $h(\bbq)$ except $d$ is achieved twice, so either $f(\bbq) = h(\bbq)$ or $f(\bbq) = h(\bbq)\setminus \{d\}$.
since every positive rational number is the sum of four squares of positive rationals, it follows that $h(\bbq)+h(\bbq)+h(\bbq)+h(\bbq)$ contains all rational numbers in $(4d,\infty)$.

For part (c), let $K$ be the quadratic extension of $\bbq$ generated by the poles of $f$. By Chebotarev density, there exists a prime $p$ such that the prime $p$ is inert in $K$ and therefore $f$ has no pole in $K$.  Thus $f(\bbq_p)$ is bounded, and by Lemma~\ref{bounds}, $f(\bbq)$ is not a virtual base.
\end{proof}

\end{document}